\documentclass[aop,preprint]{imsart} 
\usepackage{amsthm,amsmath,amsfonts,amssymb}
\usepackage{mathdots}
\usepackage{graphicx}
\usepackage{bm}
\usepackage{mathtools}
\usepackage{type1cm}
\usepackage{latexsym,enumitem}
\usepackage{dsfont}
\usepackage{bbm}
\usepackage{soul}
\usepackage{epic,tikz}

\renewcommand {\epsilon}{\varepsilon}
\renewcommand {\le}{\leqslant}
\renewcommand {\ge}{\geqslant}
\renewcommand {\le}{\leqslant}

\RequirePackage[colorlinks,citecolor=blue,urlcolor=blue]{hyperref}

\newtheorem{Theorem}{Theorem}[section]
\newtheorem{Lemma}[Theorem]{Lemma}
\newtheorem{Cor}[Theorem]{Corollary}
\newtheorem{Prop}[Theorem]{Proposition}
\newtheorem{Rem}[Theorem]{Remark}

\startlocaldefs

\def\cF{\mathcal{F}}
\def\cG{\mathcal{G}}
\def\cH{\mathcal{H}}

\def\cL{\mathcal{L}}
\def\cM{\mathcal{M}}

\def\cP{\mathcal{P}}

\def\cS{\mathcal{S}}

\def\Erw{\mathbb{E}}

\def\N{\mathbb{N}}
  
\def\Prob{\mathbb{P}} 

\def\R{\mathbb{R}}      
\def\bbS{\mathbb{S}}

\def\U{\mathbb{U}}
\def\V{\mathbb{V}}

\def\Z{\mathbb{Z}}

\def\eps{\varepsilon}

\def\vth{\vartheta}

\def\1{\mathbf{1}}
\def\3{{\ss}}

\def\eqdist{\stackrel{d}{=}}

\def\wh{\widehat}
\def\wtil{\widetilde}

\def\dd{\mathrm{d}}

\def\IRg{\R_{\scriptscriptstyle >}}
\def\IRge{\R_{\scriptscriptstyle\geqslant}}

\def\BRW{\textsf{BRW}}

\newcommand{\ind}[1]{\1_{\{#1\}}}

\allowdisplaybreaks[4] 

\endlocaldefs

\begin{document}

\begin{frontmatter}

\title{A simple method to find all solutions to the functional equation of the smoothing transform}
\runtitle{The functional equation of the smoothing transform}

\begin{aug}
\author{\fnms{Gerold}  \snm{Alsmeyer}\corref{}\thanksref{t1}\ead[label=e1]{gerolda@uni-muenster.de}}
\and
\author{\fnms{Bastien} \snm{Mallein}\corref{}\thanksref{t2}\ead[label=e2]{mallein@math.univ-paris13.fr}}
    \thankstext{t1}{Funded by the Deutsche Forschungsgemeinschaft (DFG) under Germany's Excellence Strategy EXC 2044--390685587, Mathematics M\"unster: Dynamics--Geometry--Structure.}
   \thankstext{t2}{LAGA, UMR 7539, Universit\'e Paris 13 - Sorbonne Paris Cit\'e, Universit\'e Paris 8. Partially funded by the ANR project 16-CE93-0003 (ANR MALIN).}

\runauthor{G.~Alsmeyer and B.~Mallein}
  \affiliation{University of M\"unster and Universit\'e Paris 13}
  
\address{G.~Alsmeyer\\Inst.~Math.~Stochastics,\\ Department
of Mathematics\\ and Computer Science\\ University of M\"unster\\ Orl\'eans-Ring 10, D-48149\\ M\"unster, Germany\\
          \printead{e1}\\
          }
          
          \address{B.~Mallein\\LAGA, UMR 7539\\ Universit\'e Paris 13 - Sorbonne Paris Cit\'e\\ Universit\'e Paris 8\\ 99 avenue Jean-Baptiste Cl\'ement\\ 93430 Villetaneuse, France\\
          \printead{e2}
          }
\end{aug}

\begin{abstract}
Given a nonincreasing null sequence $T = (T_j)_{j \ge 1}$ of nonnegative random variables satisfying some classical integrability assumptions and $\Erw(\sum_{j}T_{j}^{\alpha})=1$ for some $\alpha>0$, we characterize the solutions of the well-known functional equation 
$$ f(t)\,=\,\textstyle\Erw\left(\prod_{j\ge 1}f(tT_{j})\right),\quad t\ge 0, $$ 
related to the so-called smoothing transform and its min-type variant. In order to do so within the class of nonnegative and nonincreasing functions, we provide a new three-step method whose merits are that\\
(1) it simplifies earlier approaches in some relevant aspects,\\
(2) it works under weaker, close to optimal conditions
in the so-called boun\-dary case when $\Erw\big(\sum_{j\ge 1}T_{j}^{\alpha}\log T_{j}\big)=0$,\\
(3) it can be expected to work as well in more general setups like random environment.

At the end of this article, we also give a one-to-one correspondence between those solutions that are Laplace transforms and thus correspond to the fixed points of the smoothing transform and certain fractal random measures. The latter are defined on the boundary of a weighted tree related to an associated branching random walk.

\end{abstract}



\end{frontmatter}

\section{Introduction}\label{sec:intro}

Given a nonincreasing null sequence $T=(T_{j})_{j\ge 1}$ of nonnegative random variables, the mapping $f\mapsto\Erw\prod_{j\ge 1}f(tT_{j})$ for suitable functions $f:\IRge\to\R$ with $\IRge=[0,\infty)$ is called the smoothing transform and any $f$ satisfying
\begin{equation}\label{eq:functional FPE}
f(t)\ =\ \Erw\Bigg(\prod_{j\ge 1}f(tT_{j})\Bigg)
\end{equation}
a fixed point of this mapping. The problem of identifying all fixed points within certain function classes, here Laplace transforms of probability measures on $\IRge$ or, more generally, survival functions of nonnegative random variables, has been dealt with in a host of articles such as \cite{DurLig:83,Liu:98,BigKyp:05}, with most general results obtained in \cite{AlsBigMei:12}. The last reference should also be consulted for a more detailed account of the earlier literature and for further background information.

\vspace{.1cm}

The existence of fixed points of the smoothing transform has been studied mostly under the following standard conditions which are also assumed throughout this article: First,
\begin{equation}
\Erw\Bigg(\sum_{j\ge 1}\1_{\{T_{j}>0\}}\Bigg)\ >\ 1,\label{eq:(T1)}
\end{equation}
which ensures that the product in \eqref{eq:functional FPE} remains nonempty with positive probability upon iterating the fixed point equation; second, the existence of $\alpha > 0$ such that
\begin{equation}\label{eqn:alphaDef}
\Erw\Bigg(\sum_{j\ge 1} T_{j}^{\alpha}\Bigg)\,=\,1,
\end{equation}
and third that
\begin{equation}\label{eq:derivative at alpha}
\Erw\left(\sum_{j\ge 1}T_{j}^{\alpha}\log T_{j}\right)\,\le\,0.
\end{equation}
We mention that nonnegative solutions to \eqref{eqn:smoothingTransform} may still exist if \eqref{eqn:alphaDef} for some $\alpha>0$ is weakened to 
\begin{equation}\label{eqn:liu}
\Erw\Bigg(\sum_{j\ge 1}T_{j}^{\beta}\Bigg)\ \le\ 1\quad\text{for some }\beta\in [0,1].
\end{equation}
This was shown by Liu \cite[Thm.~1.1]{Liu:98} using a truncation argument.

\vspace{.1cm}
The aim of this article is to introduce a new and simple method that provides the characterization of all fixed points of the smoothing transform within the framework just given and some extra integrability conditions. Not working under most general conditions as in \cite{AlsBigMei:12} allows us to simplify and streamline the presentation of our method while, on the other hand, still being able to treat the so-called boundary case (see \eqref{eqn:boundaryCase}) under close to optimal conditions and to thus improve related earlier results in \cite{BigKyp:05,AlsBigMei:12}.

\vspace{.1cm}
If $f$ is the Laplace transform of a probability law $\nu$ on $\IRge$, then it solves Eq.~\eqref{eq:functional FPE} iff $\nu$ is a distributional fixed point of the (homogeneous) smoothing transform $\bbS$ which maps $\nu$ to the law of the random variable $\sum_{j\ge 1}T_{j}X_{j}$, where $X,X_{1},X_{2},\ldots$ denote i.i.d.~random variables with common law $\nu$ and independent of $T$. Hence, Eq.~\eqref{eq:functional FPE} corresponds to $\bbS(\nu)=\nu$ or, equivalently, 
\begin{equation}\label{eqn:smoothingTransform}
X\ \eqdist\ \sum_{j\ge 1}T_{j}X_{j}
\end{equation}
where $\eqdist$ means equality in distribution.

\vspace{.1cm}
In the case when $f$ is the (left continuous) survival function of a probability distribution $\nu$ on $\IRge$, viz. $f(t)=\nu([t,\infty))$ for $t\ge 0$, the fixed-point property \eqref{eq:functional FPE} corresponds to the distributional fixed-point equation
\begin{equation}\label{eq:SFPE max transform}
X\ \eqdist\ \inf\{X_{j}/T_{j}:j\ge 1\text{ and }T_{j}>0\}
\end{equation}
with $X,X_{1},X_{2},\ldots$ as before. Here the infimum over the empty set is defined to be $\infty$.

\vspace{.1cm}
As in \cite{AlsBigMei:12}, let $\cS(\cM)$ denote the set of solutions to Eq.~\eqref{eq:functional FPE} within the class
\begin{align*}
&\cM\ =\ \left\{f:\IRge\to [0,1]:\,f\text{ is nonincreasing and left continuous},\right.\\
&\hspace{3cm}\left. f(0)=f(0+)=1\text{ and }0<f(t)<1\text{ for some }t>0\right\}
\end{align*}
which comprises all survival functions on $\IRge$ as well as its subclass $\cL$ of Laplace transforms of probability measures on $\IRge$, ruling out only the trivial solutions $f\equiv 1$ and $f=1_{\{0\}}+q\,\1_{\IRg}$, where $q$ denotes the extinction probability of the associated branching random walk (\BRW) described below. Note that the last fact entails $\cS(\cL)\subset\cS(\cM)$. In order to determine $\cS(\cM)$ (and thus $\cS(\cL)$), which has already been done in \cite{AlsBigMei:12}, we provide here a new approach that works under relaxed conditions on the random sequence $T$ and also considerably simplifies some of the key arguments used in \cite{AlsBigMei:12}. This novel approach consists of three steps that will be outlined further below, after the introduction of some necessary notation, further background information and the most important assumptions.

\vspace{.1cm}
It is well-known that the fixed points of the smoothing transform satisfying \eqref{eqn:alphaDef} are intimately related to the essentially unique fixed-point of the modified smoothing transform
\begin{equation}\label{eqn:modifiedSmoothing}
\bar{X}\ \eqdist\ \sum_{j\ge 1}T_{j}^\alpha \bar{X}_{j},
\end{equation}
where the $\bar{X}_j$ are i.i.d.~copies of $\bar{X}$ (see also Remark \ref{rem:alpha=1} below). Here we work under conditions ensuring that $\bar{X}$ can be obtained as the limit of a martingale of an associated \BRW\ that we describe next. Let $\V= \bigcup_{n\ge 0}\N^{n}$ be the Ulam-Harris tree of finite integer words, with the convention that $\N^{0}=\{\varnothing\}$ equals the set containing the empty word (and root of the tree). As common, we use $v_{1}\ldots v_{n}$ as shorthand for $v=(v_{1},\ldots,v_{n})$, $|v|=n$ for its length, and $uv$ for the concatenation of two vertices $u,v\in\V$. The restriction of $v$ to its first $j$ coordinates, thus its ancestor at level $j$, is denoted $v(j)$, thus $v(j)=v_{1}\ldots v_{j}$. We set $\partial \V = \N^\N$, which represents the boundary of the tree $\V$.

\vspace{.1cm}
Now let $(T_{j}^{v})_{j\ge 1}$ be i.i.d.~copies of $T$ for any $v\in\V$ and define the multiplicative \BRW\ (also called weighted branching process) as the random map
$$ L:\V\to \IRge,\quad v\ \mapsto\ L(v)\,:=\,\prod_{j=1}^{|v|}T^{v(j-1)}_{v_{j}}. $$
Assumption \eqref{eq:(T1)} entails that $L$ forms a supercritical \BRW, thus with positive probability, for any $n \in \N$ there exists a vertex $v\in\V_{n}:=\{u:|u|=n\}$ such that $L(v)>0$. This event is called the survival set of the \BRW. Thinking of $T_{j}^{v}$ as a weight attached to the edge connecting vertex $v$ with its child $vj$, we see that $L(v)$ equals the total weight of the unique shortest path from the root $\varnothing$ (with $L(\varnothing):=1$) to $v$ obtained by multiplying the edge weights along this path. The natural filtration of the \BRW, reflecting its genealogical structure, is defined by
$$ \cF_{n} = \sigma\big(L(v),\,|v|\le n\big),\quad n\in\N_{0}. $$
There is a deep and meaningful relationship between the fixed points of the smoothing transform and the \BRW\ just introduced. This will be further explained in Section~\ref{sec:measure}. 

\vspace{.1cm}
A second fundamental assumption for the existence of nontrivial solutions to \eqref{eq:functional FPE}, though not necessary as shown in \cite{Liu:98} and already mentioned, is the existence of a minimal positive $\alpha$, called \emph{characteristic exponent of $T$} in \cite{AlsBigMei:12}, such that \eqref{eqn:alphaDef} holds. Then these solutions can be expressed in terms of an associated martingale limit. Namely, by the branching property of the \BRW, the process
\begin{equation}\label{eq:Biggins martingale}
W_{n}\ :=\ \sum_{|v|=n} L(v)^{\alpha},\quad n\ge 0
\end{equation}
constitutes a nonnegative martingale. It was shown by Biggins \cite{Biggins:77} (see also~\cite{Lyons:97} for a simpler proof of his result) that $W_{n}$ converges a.s.~and in $L^{1}$ to a nondegenerate limit $W$ provided that, additionally,
\begin{gather}\label{eqn1:regular case}
\Erw\Bigg(\sum_{j\ge 1}T_{j}^{\alpha}\log T_{j}\Bigg)\,<\,0
\shortintertext{and}
\Erw\left(W_{1} \log W_{1}\right)\,<\,\infty\label{eqn2:regular case}
\end{gather}
hold. Alsmeyer and Iksanov \cite{AlsIks:09} further proved that these conditions are indeed necessary and sufficient whenever $\Erw\big(\sum_{j\ge 1}T_{j}^{\alpha}\log T_{j}\big)$ is well-defined. The law of $W$ forms the unique (up to scaling) solution to the SFPE \eqref{eqn:modifiedSmoothing}.

\begin{Rem}\label{rem:senetaHeyde}
Under more general conditions, for instance, when replacing \eqref{eqn:alphaDef} with \eqref{eqn:liu}, the a.s.~limit of the martingale $(W_n)_{n\ge 0}$ may be $0$, but yet a regularly varying sequence $(a_n)_{n\ge 0}$, called Heyde-Seneta norming, exists such that $a_n W_n$ converges a.s. to a non-degenerate solution of \eqref{eqn:modifiedSmoothing}.
\end{Rem}

Condition \eqref{eqn1:regular case} is obviously more restrictive than Condition \eqref{eq:derivative at alpha}, for it requires $\Erw(\sum_{j\ge 1}T_{j}^{\alpha}\log T_{j})$ to be strictly negative. We call this situation the \emph{regular case}, as opposed to the more difficult \emph{boundary case},  where this expectation is zero and which will be discussed further below. Assuming \eqref{eqn1:regular case} and \eqref{eqn2:regular case} (naturally besides \eqref{eq:(T1)} and \eqref{eqn:alphaDef})  allows the construction of a natural solution to the SFPE. Moreover, we will see that under one more integrability condition, namely
\begin{equation}\label{eqn:finitemean}
\Erw\Bigg(\sum_{j\ge 1}T_{j}^{\alpha}\log T_{j}\Bigg)\,>\,-\infty,
\end{equation}
the set $\cS(\cM)$, and thus all fixed points of Eqs.~\eqref{eqn:smoothingTransform} and \eqref{eq:SFPE max transform}, can be determined quite easily. The structure of the fixed points depends on the lattice-type of $T$. As in \cite{AlsBigMei:12}, $T$ is called geometric with span $r$ or $r$-geometric if $r>1$ is the maximal number such that
\begin{equation}\label{eqn:lattice}
\Prob\left(T_{j} \in\{r^{n}:n\in\Z\}\text{ for all }j\ge 1\right)\,=\,1,
\end{equation}
and nongeometric if no such $r$ exists. A function $h:\IRge\to\IRg$ is called multiplicatively $r$-periodic if $h(rt)=h(t)$ for all $t\ge 0$. Given $r>1$, let $\cH_{r}$ denote the set of all such functions such that $t\mapsto h(t)t^{\alpha}$ is nondecreasing for $\alpha$ satisfying \eqref{eqn:alphaDef}. Let also $\cH_{1}$ be the the set of positive constant functions on $\IRge$. In order to be able to describe $\cS(\cL)$, we put $\cP_{1}=\cH_{1}$ and denote by $\cP_{r}$ for $r>1$ the set of $h\in\cH_{r}$ such that $t^{\alpha}h(t)$ has a completely monotone derivative. When $\alpha=1$, the latter requirements force $h$ to be constant, thus $\cP_{r}=\cP_{1}$. These classes were first introduced in \cite{DurLig:83}.

\begin{Theorem}\label{thm:mainNonBoundary}
Let $T$ satisfy \eqref{eq:(T1)}, \eqref{eqn:alphaDef}, \eqref{eqn1:regular case}, \eqref{eqn2:regular case}, \eqref{eqn:finitemean} and $r\ge 1$ denote its span. Then the elements of $\cS(\cM)$ are the functions
$$ f(t)\ =\ \Erw(\exp(-h(t)t^{\alpha}W)),\quad t\ge 0 $$
with left continuous $h\in\cH_{r}$ and $W = \lim_{n \to \infty}W_n$ a.s. Moreover, $\cS(\cL)$ is nonempty only if $0<\alpha\le 1$ in which case it consists of all $f\in\cS(\cM)$ with $h\in\cP_{r}$.
\end{Theorem}

This result was derived in \cite{AlsBigMei:12} by first showing the slow variation of the function $(1-\Erw e^{-tW})/t$ at $t=0$ (Thm.~3.1) and then, given any $f\in\cS(\cM)$, the same property for $(1-f(t))/t^{\alpha}h(t)$ for a suitable $h\in\cH_{r}$ (Thms. 3.2 and 3.3). Our method of proof does not need these intermediate results, but rather replaces them by a simpler boundedness argument.

\begin{Rem}\label{rem:alpha=1}\rm
If $\alpha=1$, then Eq.~\eqref{eqn:smoothingTransform} becomes Eq.~\eqref{eqn:modifiedSmoothing}, and the only fixed point, modulo positive multiplicative constants, is given by the limit $W$ of the Biggins martingale defined in \eqref{eq:Biggins martingale}. This fixed point is \emph{endogenous} in the sense of Aldous and Bandyopadhyay \cite{AldBan:05} which means that $W$ can be constructed as a measurable function of the \BRW\ or, equivalently, that it is $\cF_{\infty}$-measurable, where $\cF_{\infty}:=\sigma(\cF_{n},\,n\ge 0)$. If $\alpha < 1$, there are no endogenous fixed points.
\end{Rem}

\begin{Rem}\rm
Although working here under conditions that ensure the convergence of the additive martingale, our method can be adapted to the case of Heyde-Seneta-type norming mentioned in Rem.~\ref{rem:senetaHeyde}. Then one has to replace the function $t^\alpha$ by a suitable regularly varying function of index $\alpha$ in the proof of Theorem \ref{thm:mainNonBoundary}.
\end{Rem}

Nontrivial solutions to Eq.~\eqref{eq:functional FPE} also exist in the boundary case when
\begin{equation}\label{eqn:boundaryCase}
\Erw\Bigg(\sum_{j\ge 1}T_{j}^{\alpha}\log T_{j}\Bigg) = 0,
\end{equation}
holds in the place of \eqref{eqn1:regular case}. In this case, the martingale limit $W$ is a.s.~zero. However, assuming \eqref{eqn:boundaryCase}, the process
\begin{equation}\label{eqn:derivativeMartingale}
 Z_{n}\ =\ \sum_{|v|=n}L(v)^{\alpha}(-\log L(v)), \quad n \ge 0
\end{equation}
is also a martingale, known as the derivative martingale. Quite general, but nonoptimal  conditions for the almost sure convergence of $Z_{n}$ to a nonnegative and nondegenerate limit were provided by Biggins and Kyprianou~\cite{BigKyp:04}. Later A{\"{\i}}d\'ekon \cite{Aidekon:13} determined necessary conditions, that Chen \cite{Chen:15} found sufficient. More precisely, if
\begin{equation}\label{eqn:ncsDerivative}
\Erw\Bigg(\sum_{j\ge 1}T_{j}^{\alpha}\log^{2}T_{j}\Bigg) \in (0,\infty), \quad \Erw W_{1}\log^{2}W_{1}\,+\,\Erw\wtil{X}\log\wtil{X}\ <\ \infty,
\end{equation}
where $\wtil{X}:=\sum_{j\ge 1}T_{j}^{\alpha}\log_{+}\!T_{j}$, then
\begin{equation}\label{eqn:derivativeLimit}
Z\, :=\, \lim_{n\to\infty} Z_{n} \quad \text{a.s.}
\end{equation}
exists and is almost surely positive on the survival set of the \BRW.

\vspace{.1cm}
Biggins and Kyprianou \cite{BigKyp:05} and Alsmeyer et al. \cite{AlsBigMei:12} proved analogs of Theorem~\ref{thm:mainNonBoundary} in the boundary case, where the limit $Z$ of the derivative martingale replaces the limit $W$ of the additive martingale. Their proofs required some exponential integrability condition, see (A) in \cite{BigKyp:05} and (A4b) in \cite{AlsBigMei:12}. Using the same techniques as in the proof of Theorem~\ref{thm:mainNonBoundary}, we can extend their results without any additional assumption other than those ensuring the nondegeneracy of the derivative martingale. This can be seen as a major advantage of our approach.

\begin{Theorem}\label{thm:mainDerivative}
Let $T$ satisfy \eqref{eq:(T1)}, \eqref{eqn:alphaDef}, \eqref{eqn:boundaryCase} and \eqref{eqn:ncsDerivative}. Then all assertions of Theorem~\ref{thm:mainNonBoundary} remain valid when replacing $W$ with $Z$ in the definition of the solutions $f$.
\end{Theorem}

\begin{Rem}\rm
The following example shows that one can easily provide sequences $(T_{n})_{n\ge 1}$ to which Theorem \ref{thm:mainDerivative} applies but which are not covered by the corresponding result in \cite{AlsBigMei:12} because they fail to satisfy the condition (A4b) in that article. Define
$$ T_j\ =\ X (j/\log^3 (j+1))^{1/\alpha},\quad j\ge 1, $$
and let $X$ be any nonnegative random variable such that \eqref{eqn:alphaDef}, \eqref{eqn:boundaryCase} and \eqref{eqn:ncsDerivative} hold. Then we obviously have $\Erw(\sum T_j^{\beta}) = \infty$ for all $\beta < \alpha$.
\end{Rem}

Although the proofs of Theorems~\ref{thm:mainNonBoundary} and~\ref{thm:mainDerivative} need different estimates, they follow the same three-step scheme that we now outline (in the regular case) and believe to work also in more general situations. Given any solution $f\in\cS(\cM)$, it is easily checked that, for each $t\ge 0$,
\begin{align*}
M_{n}(t):=\prod_{|v|=n}f(tL(v)),\quad n\ge 0
\end{align*}
constitutes a positive bounded product martingale whose limit
\begin{align*}
M(t)\ :=\ \lim_{n\to\infty}\prod_{|v|=n}f(tL(v))
\end{align*}
exists a.s.~and in $L^{1}$. These martingales are called in \cite{AlsBigMei:12} the \emph{disintegration} of the fixed point $f$.
\begin{description}\itemsep2pt
\item[1. Tameness:] Using the convergence of the disintegration along a sequence of stopping lines (see Section~\ref{sec:brw}), the first step is to show that any nondegenerate fixed point $f$ must satisfy
$$ \limsup_{t\to0} \frac{-\log f(t)}{t^{\alpha}}\, <\, \infty. $$
\item[2. Harmonic analysis:] This property enables us to derive that $-\log M(t)$ is an integrable random variable with
$$  F(t)\,:=\,\Erw(-\log M(t))\,=\,\Erw\Bigg(\sum_{|v|=1}F(tL(v))\Bigg)\,\le\,C t^{\alpha} $$
for all $t>0$ and suitable $C\in (0,\infty)$. The shown equality can be translated as follows: the function $G(x):=e^{\alpha x}F(e^{-x})$ defines a bounded harmonic function of a certain random walk associated with the \BRW\ (see Section \ref{sec:brw}). By a Choquet-Deny-type lemma, we then deduce that $G$ is constant on the subgroup generated by the walk.
\item[3. Identification of $M(t)$:] It follows by the previous step that $F$ is of the form $F(t)=h(t)t^{\alpha}$ for some $h\in\cH_{r}$, $r$ the span of $T$. To find the value of $M(t)$, we finally observe that
\begin{align*}
\log M(t)\ &=\ \lim_{n\to\infty}\Erw\left(\log M(t)|\cF_{n}\right)\\
&=\ -\lim_{n\to\infty}\sum_{|v| = n}F(tL(v))\ =\ -h(t)t^{\alpha} W
\end{align*}
and thus complete the proof of the main theorem as $f(t)=\Erw e^{\log M(t)}$.
\end{description}

In the boundary case, this three-step method requires the identification of a harmonic function of at most linear growth for a killed random walk. This is done in Proposition~\ref{prop:uniqueness H(x)}, where we prove that there is only one such function modulo scalars, namely the renewal function of a random walk. This result, which may also be of independent interest, forms a generalization of a result by Spitzer \cite[Thm.~E3, p.~332]{Spitzer:76} and shows that the Martin boundary of a centered random walk with finite variance and conditioned to stay positive reduces to a single point.

\vspace{.1cm}
We devote the next section to some classical \BRW\ tools and then prove Theorem~\ref{thm:mainNonBoundary} in Section~\ref{sec:nonBoundary}. Before turning to the more difficult proof of Theorem~\ref{thm:mainDerivative} in Section~\ref{sec:derivative}, we need to show in Section~\ref{sec:RW negative halfline} a Choquet-Deny-type result asserting that any right-continuous and at most linearly growing  harmonic function of a centered random walk killed upon entering $\IRge$ equals the Tanaka solution (see \eqref{eq:Tanaka function}) up to a constant, or a periodic function in the lattice case. Finally, in Section~\ref{sec:measure} we briefly describe a one-to-one connection between the solutions of \eqref{eqn:smoothingTransform} and certain fractal random measures on the boundary of the weighted tree related to the \BRW.

\section{Preliminary results for the classical branching random walk}
\label{sec:brw}

This section collects some well-known tools in the study of \BRW's that will be needed later on, namely the many-to-one lemma and some facts about stopping lines.

The many-to-one lemma is a widely known result, which can be traced back at least to Peyri\`ere \cite{Peyriere:74} and Kahane and Peyri\`ere \cite{KahanePey:76}. It links additive moments of the \BRW\ to random walk estimates. Consider a zero-delayed random walk $(S_{n})_{n\ge 0}$ with increment distribution specified as
\begin{equation}\label{eqn:defManytooneRW}
\Erw g(S_{1})\ =\ \Erw\Bigg(\sum_{j\ge 1}T_{j}^{\alpha}g(-\log T_{j})\Bigg).
\end{equation}
for nonnegative measurable $g$.

\begin{Lemma}[Many-to-one lemma]\label{lem:manytoone}
For all $n\ge 0$ and all nonnegative measurable functions $g$, we have 
\[\Erw g(S_{1},\ldots,S_{n})\ =\ \Erw\Bigg(\sum_{|v|=n} L(v)^{\alpha}g(-\log L(v(j)), j \le n)\Bigg).\]
\end{Lemma}

The result can be thought of as a first step towards the spinal decomposition due to Lyons \cite{Lyons:97} that describes the law of the \BRW\ when size-biased by the martingale $(W_{n})_{n\ge 0}$ as a \BRW\ with a designated path, called spine, along which offspring particles have displacement law defined by \eqref{eqn:defManytooneRW}.

As moments of functionals of the $tL(u)$ will often be considered hereafter, it is convenient to define $\Prob_{t}$ as the law of the above random walk $(S_{n})_{n \ge 0}$ when its delay equals $S_{0} = -\log t$. In other words, the random walk is starting from $-\log t$ under $\Prob_{t}$, and its laws under $\Prob$ and $\Prob_{1}$ coincide. With this notation, Lemma~\ref{lem:manytoone} can be rewritten as
\begin{equation}
  \label{eqn:startingPointmanytoone}
  \Erw\Bigg(\sum_{|v|=n} L(v)^{\alpha}g( tL(v(j)), j \le n)\Bigg) \ =\ \Erw_t g(e^{-S_{1}},\ldots,e^{-S_{n}}).
\end{equation}

We now recall some facts about stopping lines, in fact so-called very simple stopping lines, a name coined by Biggins and Kyprianou \cite[p.~557]{BigKyp:04}. A line is a set $\mathcal{L}\subset\V$ satisfying the two following assumptions:
\begin{equation}\label{eqn:line}
\begin{split}
    &\hspace{.5cm}\forall u,v \in \mathcal{L}: u\preceq v\ \Rightarrow\ u=v,\\
    &\forall v\in \partial\V: v(n)\in\mathcal{L}\text{ for some }n\in\N.
\end{split}
\end{equation}
In other words, a line is a minimal set separating the root $\varnothing$ from the boundary $\partial\V$. In \BRW's, stopping lines take the role of stopping times for random walks. In particular, a very simple stopping line is a random line such that for all $v\in\V$,
\begin{equation}\label{eqn:stoppingLine}
\{v\in\mathcal{L}\}\ \in\ \sigma(L(v(j)),j\le |v|).
\end{equation}
In other words, whether or not a vertex $v$ belongs to the line depends only on the weights of the tree along the unique path from the root to $v$.

\vspace{.1cm}
In this article, only the following first passage lines will be of interest. For all $a>0$, we set
\begin{equation}
\Upsilon_{a}\ :=\ \left\{v\in\V:L(v(j))\ge a\text{ for }j<|v|\text{ and }L(v)<a\right\}.
\end{equation}
Note that $\lim_{n\to\infty}\sup_{|v|=n}L(v)=0$ a.s. under the assumptions of our two theorems. Therefore, $\Upsilon_{a}$ is a well-defined very simple stopping line for any $a>0$ and consists of all particles entering the interval $[0,a]$ for the first time. Biggins and Kyprianou \cite{BigKyp:04} proved that a theorem similar to the optional stopping theorem holds for the additive martingale of the \BRW. We state and use here a simplified version of their result. Further defining
\begin{align}\label{eq:def first passage filtration}
\cG_{a}\ :=\ \sigma\left(L(v(j)),\,j\le |v|,\,v\in\Upsilon_{a}\right), 
\end{align}
note that $(\cG_{e^{-t}})_{t\ge 0}$ forms a filtration for the \BRW.
The following result is a version of the many-to-one lemma along the stopping lines $\Upsilon_{a}$ and can also be found for example in \cite{Kyprianou:00}. We give a proof here for sake of completeness.

\begin{Lemma}\label{lem:manytooneStopped}
For all $a \in (0,1]$ and all measurable nonnegative functions $g$, we have
\begin{align*}
\Erw\Bigg(\sum_{v\in\Upsilon_{a}}L(v)^{\alpha}g(L(v(j)),j \le |v|)\Bigg)\ =\  \Erw g\big(e^{-S_j},j \le\sigma(-\log a)\big),
\end{align*}
where  $\sigma(b):=\inf\{n\ge 0: S_{n}>b\}$ and $(S_{n})_{n\ge 0}$ equals the random walk with increment law defined by \eqref{eqn:defManytooneRW}.
\end{Lemma}

\begin{proof}
The result is obtained by a decomposing the cutting line genera\-tion-wise and the applying the many-to-one lemma, viz.
\begin{align*}
&\Erw\left(\sum_{v\in \Upsilon_{a}} L(v)^{\alpha}g( L(v(j)),j\le |v|)\right)\\
&=\ \sum_{n\ge 0} \Erw\Bigg(\sum_{|v|=n}\1_{\{v\in\Upsilon_{a}\}}L(v)^{\alpha} g( L(v(j)),j \le n)\Bigg)\\
&=\ \sum_{n\ge 0}\Erw g(e^{-S_{1}},\ldots,e^{-S_{n}})\1_{\{\sigma(-\log a)=n\}}\\
&=\ \Erw g\big(e^{-S_{j}}, j \le\sigma(-\log a)\big).
\end{align*}
This completes the proof.
\end{proof}

\section{The regular case: proof of Theorem~\ref{thm:mainNonBoundary}}\label{sec:nonBoundary}

Given any solution $f\in\cS(\cM)$, recall that, for each $t\ge 0$, the disintegration
\begin{align*}
M_{n}(t):=\prod_{|v|=n}f(tL(v)),\quad n\ge 0
\end{align*}
constitutes a positive bounded product martingale whose limit $M(t)$ exists a.s.~and in $L^{1}$. We start by proving the tameness of any solution $f\in\cS(\cM)$.

\begin{Lemma}\label{lem:tameness}
Under the assumptions of Theorem~\ref{thm:mainNonBoundary}, any $f\in\cS(\cM)$ satisfies
\begin{equation}\label{eqn:tameness}
\sup_{0<t\le 1}\frac{-\log f(t)}{t^{\alpha}}\ \le\ C
\end{equation}
for some $0<C<\infty$.
\end{Lemma}

\begin{proof}
Assuming $\limsup_{t\to 0}\frac{-\log f(t)}{t^{\alpha}} = \infty$, we will derive that for all $t > 0$ we have $f(t)\le\Prob(W=0)$. Since $\Prob(W=0)<1$, this contradicts the property $f(0+)=1$.

By the stated assumption, there exists a decreasing null sequence $(t_{n})_{n \le 1}$ such that
\begin{align*}
\frac{-\log f(t_{n})}{t_{n}^{\alpha}}\ \ge\ n^{2}\quad\text{for all }n\ge 1.
\end{align*}
Setting $c_{n}:=n^{1/\alpha}$ and observing that $t \mapsto -\log f(t)$ and $t \mapsto t^{\alpha}$ are both nonnegative and nondecreasing functions, we find that
\begin{align*}
\frac{-\log f(s)}{s^{\alpha}}\ \ge\ \frac{-\log f(t_{n})}{(c_{n}t_{n})^{\alpha}}\ \ge\ n\quad\text{for all }s \in [t_{n}, c_{n} t_{n}].
\end{align*}
Therefore, we can define decreasing null sequences $(\vth_{n})_{n\ge 1}$ and $(\rho_{n})_{n\ge 1}$ such that
\begin{equation}\label{eqn:extractedSequence}
\frac{-\log f(s)}{s^{\alpha}}\ \ge\ n\quad\text{for all }s\in [\rho_{n}\vth_{n}, \vth_{n}].
\end{equation}

\vspace{.1cm}
We will now bound the conditional mean of $M(t)$ given $\cG_{\vth_{n}/t}$, where $(\cG_{\vth_{n}/t})_{n\ge 1}$ is the first passage filtration defined in \eqref{eq:def first passage filtration}. By dominated convergence,
\begin{align*}
\Erw\left(M(t)\middle|\cG_{\vth_{n}/t} \right)\ =\ \lim_{m\to\infty} \Erw\left( M_{m}(t)\middle|\cG_{\vth_{n}/t}\right),
\end{align*}
and the branching property of the \BRW\ implies
\begin{align*}
\Erw\left(M_{m}(t) \middle|\cG_{\vth_{n}/t}\right)\ = \exp\Bigg(\sum_{\substack{v\in\Upsilon_{\vth_{n}/t}\\|v| \le m}} \log f(tL(v))\,+\,\sum_{\substack{|v|=m\\ tL_{*}(v)\ge\vth_{n}}}f(tL(v))\Bigg),
\end{align*}
where $L_{*}(v):=\min_{0\le k\le|v|}L(v(k))$. Hence, using $\sup_{v\in \Upsilon_{\vth_{n}/t}}|v|<\infty$ a.s., we infer upon letting $m\to\infty$
\begin{align}
\Erw\left(M(t)\middle|\cG_{\vth_{n}/t} \right)\ &=\ \exp\left(\sum_{v\in \Upsilon_{\vth_{n}/t}} \log f(tL(v))\right)\nonumber\\
&\le\ \exp\Bigg(-n\sum_{v\in\Upsilon_{\vth_{n}/t}}(tL(v))^{\alpha}\1_{\{tL(v)\ge\rho_{n}\vth_{n}\}}\Bigg),\label{eqn:firststep}
\end{align}
where we have bounded $-\log f(tL(v))$ by $0$ if $tL(v)\not\in [\rho_{n}\vth_{n},\vth_{n}]$, and with the help of \eqref{eqn:extractedSequence} otherwise.

\vspace{.1cm}
On the other hand, by another use of the branching property of the \BRW, we obtain, for all $a>0$ and $m\in\N$,
$$ \Erw\left(W_m\middle|\cG_{a} \right)\ =\ \sum_{\substack{v\in\Upsilon_{a}\\|v| \le m}} L(v)^{\alpha} + \sum_{\substack{|v|=m\\ L_{*}(u) > a}} L(v)^{\alpha}\quad\text{ a.s.} $$
and thus $\Erw\left(W\middle|\cG_{a} \right)\ =\ \sum_{v\in\Upsilon_{a}} L(v)^{\alpha}$ a.s.~as $m\to\infty$. Now let $a\to 0$ and use $\mathcal{F}_\infty = \bigvee_{a > 0}\cG_{a}$ to infer
\begin{equation}\label{eqn:1}
\lim_{a\to0}\sum_{v\in\Upsilon_{a}} L(v)^{\alpha}\ =\ W \text{ a.s.}
\end{equation}

Next, Lemma~\ref{lem:manytooneStopped} provides us with
\begin{align*}
\Erw\left(\sum_{v\in\Upsilon_{\vth_{n}/t}}(tL(v))^{\alpha}\1_{\{tL(v)<\rho_{n}\vth_{n}\}}\right)\,=\,t^{\alpha}\,\Prob\bigg(S_{\sigma(-\log(\vth_{n}/t))}> -\log\Big(\frac{\rho_{n}\vth_{n}}{t}\Big)\hspace{-4pt}\bigg),
\end{align*}
where $\sigma(a) = \inf\{ n \le 0 : S_{n}>a\}$ should be recalled. Use \eqref{eqn1:regular case} and \eqref{eqn:finitemean} to infer
$$ \Erw S_{1}\ =\ -\,\Erw\Bigg(\sum_{j\ge 1}T_{j}^{\alpha}\log T_{j}\Bigg)\ \in\ (0,\infty). $$
As a consequence, the family $\{S_{\sigma(-\log s)}+\log s:0<s\le 1\}$ of overshoots of the random walk is tight (see e.g.~Gut \cite[Thm.~3.10.3]{Gut:09}) which in turn implies
\begin{align*}
\lim_{n\to\infty} \sum_{v\in\Upsilon_{\vth_{n}/t}}(tL(v))^{\alpha}\1_{\{tL(v)<\rho_{n}\vth_{n}\}}\ =\ 0 \quad \text{in probability}.
\end{align*}
as $\rho_{n}\to0$. In combination with \eqref{eqn:1}, this further entails that
\begin{align*}
\lim_{n\to\infty} \sum_{v\in\Upsilon_{\vth_{n}/t}}(tL(v))^{\alpha}\1_{\{tL(v)\ge\rho_{n}\vth_{n}\}}\ =\ W \quad \text{in probability},
\end{align*}
and by combining the last conclusion with \eqref{eqn:firststep}, we obtain
\begin{align*}
\liminf_{n\to\infty}\Erw\big(M(t)\big|\cG_{\vth_{n}/t}\big)\ \le\ \1_{\{W=0\}} \quad \text{a.s.}
\end{align*}
Finally using that $\left(\Erw(M(t)|\cG_{\vth_{n}/t})\right)_{n\ge 0}$ forms a bounded martingale, we infer $M(t)\le\1_{\{W=0\}}$ a.s. for all $t>0$ and thereupon the announced contradiction $f(t)=\Erw M(t)\le\Prob(W=0)$ for all $t > 0$.
\end{proof}

\begin{Rem}\rm
In the above proof, Assumption \eqref{eqn:finitemean} is only needed for the conclusion that
\begin{align*}
\sum_{v\in\Upsilon_{\vth_{n}/t}}(tL(v))^{\alpha} \1_{\{tL(v)\ge\rho_{n}\vth_{n}\}}
\end{align*}
converges to a positive random variable with positive probability. We suspect that it can be replaced with a weaker assumption while keeping the assertions of Theorem~\ref{thm:mainNonBoundary}.
\end{Rem}

With the help of Lemma~\ref{lem:tameness}, we can now identify the function defined by
\begin{equation}\label{eqn:defineH}
F(t)\ :=\ \Erw\left(-\log M(t)\right)
\end{equation}
for a given $f\in\cS(\cM)$. Doing so, we make use of the subsequent Choquet-Deny lemma, for convenience reformulated here in our setting, which identifies all bounded harmonic functions of the random walk.

\begin{Lemma}\label{lem:CD-lemma G}
Let $G : \R\to\R$ be a right-continuous bounded function satisfying
\begin{equation}\label{eq:CD-equation H}
G(x)\ =\ \Erw G(x+S_{1})
\end{equation}
for all $x\in\R$. Then $G$ $d$-periodic if $(S_{n})_{n\ge 0}$ is $d$-arithmetic, and it is constant everywhere if $(S_{n})_{n\ge 0}$ is nonarithmetic.
\end{Lemma}

\begin{proof}
Note that, possibly upon adding a constant, we can assume $G$ to be nonnegative. We denote by $\nu$ the law of $S_{1}$, and we let $\lambda$ denote the measure with density $G$ with respect to Lebesgue measure. Then \eqref{eq:CD-equation H} can be rewritten as
\[
  \lambda = \lambda \ast \nu.
\]
The Choquet-Deny lemma \cite{Choquet+Deny:60} entails that, for each $a$ in the support of $\nu$, the measure $\lambda$ is $a$-periodic. In particular, $G(x+a)=G(x)$ for Lebesgue almost all $x \in \R$. Thus using the right-continuity of $G$, this equation in fact holds for all $x \in \R$ and $a$ in the support of $\nu$. The set of periods of a right-continuous function being a closed group, we deduce that $G$ is $d$-periodic if $(S_{n})_{n\ge 0}$ is $d$-arithmetic for some $d>0$, and that $G$ is constant otherwise.
\end{proof}

We now turn to the identification of the function $F$.

\begin{Lemma}\label{lem:identificationOfH}
Given the assumptions of Theorem~\ref{thm:mainNonBoundary}, let $f\in\cS(\cM)$ with disintegration $(M_{n}(t))_{n\ge 0}$. Then there exists a function $h\in\cH_{r}$, $r$ the span of $T$, such that
\begin{equation}\label{eq:form of F}
F(t)\ =\ h(t)t^{\alpha}
\end{equation}
for all $t\ge 0$.
\end{Lemma}

\begin{proof}
Since $M(t)=\lim_{n\to\infty}M_{n}(t)$ a.s., we see that
\begin{align*}
\lim_{n\to\infty}\sum_{|v|=n}-\log f(t L(v))\ =\ -\log M(t)\quad\text{a.s.}
\end{align*}
Moreover, by Lemma~\ref{lem:tameness}, there exists $C>0$ such that
\begin{align*}
\sum_{|v|=n}-\log f(t L(u))\ \le\ C t^{\alpha} W_{n}.
\end{align*}
for all $n$ so large that $\sup_{|v|=n}L(v)\le 1$. Hence, $0 \le-\log M(t)\le C t^{\alpha}W$ follows upon letting $n\to\infty$. Since $\Erw W=1$ under the assumptions of Theorem~\ref{thm:mainNonBoundary}, we infer
\begin{equation}\label{eqn:boundedFunction}
0\,\le\,F(t)\,\le\,C t^{\alpha},
\end{equation}
where $F$ is the function defined in \eqref{eqn:defineH} (for the given $f$). One can readily check that $F$ is nondecreasing and left continuous.

\vspace{.1cm}
Put $G(x):=e^{-\alpha x}F(e^{-x})$ and use the branching property of the \BRW\  to obtain, for any $m\in\N$,
\begin{equation}\label{eqn:inlaw}
M(t)\ =\ \prod_{|u|=m}M^{(u)}(tL(u)),
\end{equation}
where $M^{(u)}(s) = \lim_{n\to\infty}\prod_{|v|=n}f(s\frac{L(uv)}{L(u)})$ are i.i.d. copies of $M$ and independent of $\cF_{m}$ for $u\in\V_{m}$. Equality in law is already enough to infer that
\begin{align*}\label{eqn:harmonic}
G(x)\ &=\  e^{-\alpha x}\,\Erw \log M(e^{-x})\ =\ e^{-\alpha x}\,\Erw\Bigg(\log \prod_{|v|=1} M^{(v)}(e^{-x}L(v))\Bigg)\\ 
&=\ \Erw\Bigg(\sum_{|v|=1} L(u)^{\alpha} G(x - \log L(v))\Bigg)\ =\ \Erw G(x + S_{1}),
\end{align*}
by the many-to-one lemma. Therefore, by \eqref{eqn:boundedFunction}, $G$ is a bounded, nonnegative and right continuous harmonic function of the random walk $(S_{n})_{n\ge 0}$, and the latter is $\log r$-arithmetic. It follows by Lemma~\ref{lem:CD-lemma G}, that $G$ is $\log r$-periodic, thus $G(x+\log r)=G(x)$ for all $x\in\R$. Equivalently, \eqref{eq:form of F} holds with $h\in\cH_{r}$ given by $h(t):=G(\log t)$ for $t>0$.
\end{proof}

\begin{Rem}\rm
The previous proof has also shown that, if \eqref{eqn1:regular case} and \eqref{eqn2:regular case} fail and thus $W_{n}\to 0$ a.s., any solution $f\in\cS(\cM)$ satisfying \eqref{eqn:tameness} must be trivial, i.e.~$f(t)=1$ for all $t\ge 0$. In particular, no nontrivial solution $f$ can satisfy \eqref{eqn:tameness} in the boundary case.
\end{Rem}

We are now ready to provide the proof of Theorem~\ref{thm:mainNonBoundary}.

\begin{proof}[Proof of Theorem~\ref{thm:mainNonBoundary}]
Given any $f\in\cS(\cM)$, we denote by $M(t)$ its disintegration and put $F(t)=\Erw(-\log M(t))$. It folows from \eqref{eqn:inlaw} that
\begin{align*}
\Erw\left(-\log M(t)\middle|\mathcal{F}_{n}\right)\ =\ \sum_{|v|=n}F(tL(v))\ \quad \text{a.s.} 
\end{align*}
for all $n\in\N$. By letting $n\to\infty$ and an appeal to Lemma~\ref{lem:identificationOfH}, we obtain
\begin{align*}
-\log M(t)\ =\ \lim_{n\to\infty}\sum_{|v|=n}F(tL(v))\ =\ h(t)t^{\alpha}\lim_{n\to\infty}W_{n}\ =\ h(t)t^{\alpha}W\quad\text{a.s.}
\end{align*}
for some $h\in\cH_{r}$, $r$ the span of $T$, and then $f(t)=\Erw M(t)= \Erw e^{-h(t)t^{\alpha} W}$. If $f\in\cS(\cL)$, we even infer $h\in\cP_{r}$ because $f$ is a Laplace transform.
\end{proof}

\section{Harmonic functions of random walks on the positive halfline}\label{sec:RW negative halfline}

We now turn to the proof of Theorem~\ref{thm:mainDerivative} and thus work under assumptions \eqref{eqn:boundaryCase} and \eqref{eqn:ncsDerivative}. In this case, instead of using the Choquet-Deny lemma, we need to identify harmonic functions of a centered random walk with finite variance, killed upon entering the nonpositive halfline. This is the content of the present section.

We recall that $(S_{n})_{n \ge 0}$ is the random walk associated with the \BRW\ by the many-to-one lemma. Since $\Erw S_{1}=0$ by \eqref{eqn:boundaryCase} and $0<\Erw S_{1}^{2} < \infty$ by \eqref{eqn:ncsDerivative}, $(S_{n})_{n\ge 0}$ is a centered random walk with finite variance. A harmonic function $G$ of the walk, killed at the first time it leaves the positive halfline $\IRg$, is a function such that $G(x) = 0$ for $x \le 0$ and
\begin{equation}\label{eq:harmonic V(x)}
G(x)\ =\ \Erw G(x+S_{1})\1_{\{x+S_{1}>0\}}\ =\ \Erw_{x}G(S_{1})\1_{\{S_{1}>0\}}.
\end{equation}
for all $x>0$.

Let us define
\begin{gather*}
\tau(a)\,:=\,\inf\{n\ge 0:S_{n}\le a\},\quad\tau\,:=\,\tau(0),
\shortintertext{and recall that}
\sigma(a)\,:=\,\inf\{n\ge 0:S_{n}>a\},\quad\sigma\,:=\,\sigma(0),
\end{gather*}
for $a\in\R$. Further, put $R_{a}:=S_{\sigma(a)}-a$, and let $(\tau_{n})_{n\ge 1}$ and $(\sigma_{n})_{n\ge 1}$ denote the sequences of weakly descending and strictly ascending ladder epochs, respectively. Note that $\Prob_{x}(\tau(a)\in\cdot)=\Prob_{0}(\tau(a-x)\in\cdot)$ for $a\le 0$ and $x\ge 0$ and recall that $\Prob=\Prob_{0}$.

Even without assuming finite variance, Tanaka \cite{Tanaka:89} obtained a solution of \eqref{eq:harmonic V(x)} defined by
\begin{equation}\label{eq:Tanaka function}
\wh{H}(x)\ :=\ \Erw_{x}\left(\sum_{k=0}^{\sigma-1}\1_{(0,x]}(S_{k})\right),\quad x>0.
\end{equation}
By the duality lemma, it also equals the renewal function of the weakly descending ladder heights $S_{n}^{*}=S_{\tau_{n}}$, $n\ge 1$, of the given walk (up to a reflection), viz.
\begin{equation*}
\wh{H}(x)\ =\ \sum_{n\ge 0}\Prob(S_{n}^{*}>-x)\ =\ \sum_{n\ge 0}\Prob(\tau^{*}(-x)>n)\ =\ \Erw\tau^{*}(-x)
\end{equation*}
for $x>0$, where $\tau^{*}(a):=\inf\{n\ge 0:S_{n}^{*}\le a\}$. Now, if $\Erw|S_{1}^{*}|<\infty$, a sufficient condition being $\Erw S_{1}^{2}<\infty$, then Wald's identity further ensures
$$ \wh{H}(x)\ =\ \frac{\Erw S_{\tau^{*}(-x)}^{*}}{\Erw S_{1}^{*}},\quad x<0, $$
and by finally observing $S_{\tau^{*}(-x)}^{*}=S_{\tau(-x)}$, we arrive at
\begin{equation}\label{eq:V(x)=ES_tau(-x)}
\wh{H}(x)\ =\ \frac{\Erw S_{\tau(-x)}}{\Erw S_{1}^{*}}\ =\ \frac{\Erw_{x}S_{\tau}-x}{\Erw S_{1}^{*}}.
\end{equation}
In other words, if $0<\Erw S_{1}^{2}<\infty$, then $\wh{H}(x)$ and $H(x):= x-\Erw_{x}S_{\tau}$ differ only by a multiplicative positive constant.

\vspace{.1cm}
An interesting aspect of this last observation is that, unlike $\wh{H}$, the function $H$ is very easily shown to be harmonic. Namely, as $\Prob(\tau(-x)\ge 1)=1$ for $x>0$ and $\Erw S_{1}=0$, we infer by a standard renewal argument
\begin{align*}
H(x)\ =\ -\Erw S_{\tau(-x)}\ &=\ \int_{\IRg}-\Erw S_{\tau(-y)}\ \Prob_{x}(S_{1}\in dy)\ =\ \Erw_{x}H(S_{1})
\end{align*}
for all $x>0$ as required.

\vspace{.1cm}
A well-known result from renewal theory asserts that 
$$ \frac{\Erw|S_{\tau^{*}(-y)}^{*}+y|}{y}\ \xrightarrow{y\to\infty}\ 0 $$
if $\Erw S_{1}^{*}<\infty$, see e.g.~\cite[Thm.~3.10.2]{Gut:09}, giving
\begin{equation}\label{eq:growth of H}
x\ \le\ H(x)\ \le\ x(1+o(1))\quad\text{as }x\to\infty.
\end{equation}
Moreover, we point out that, by definition, $H$ is right-continuous with left limits at each point.

Our main result of this section is the following Choquet-Deny-type lemma. It states that any right-continuous function of at most linear growth and satisfying \eqref{eq:harmonic V(x)} equals $H$ up to multiplication by a constant, or $d$-periodic function if the walk is $d$-arithmetic.

\begin{Prop}\label{prop:uniqueness H(x)}
Given a nontrivial, centered random walk with lattice-span $d\ge 0$ and $\Erw S_{1}^{2}<\infty$, let $G : \R\to\R$ be a right-continuous function satisfying \eqref{eq:harmonic V(x)} and $\sup_{x \le 0} |G(x)/(1+|x|)| < \infty$. Then there exists a function $\kappa$, $d$-periodic if $d>0$ and constant if $d=0$, such that
\[
G(x)\,=\,\kappa(x) H(x)\quad\text{for all }x\in\R.
\]
\end{Prop}

For centered random walks on the integer lattice $\Z$, where \eqref{eq:harmonic V(x)} must only hold for $x\in\Z$, it was already shown by Spitzer \cite[Thm.~E3, p.~332]{Spitzer:76} that there is only one positive solution to \eqref{eq:harmonic V(x)} up to positive multiples (even without additional moment assumptions). More recent work by Doney \cite[Thm.~1]{Doney:98} also considers the case when the $\Z$-valued random walk has nonzero drift.

Before proving our result, we provide some useful estimates and begin with an extension of the harmonic property of $G$ at random times.
\begin{Lemma}\label{lem:harmonicStopped}
Under the assumptions of Proposition \ref{prop:uniqueness H(x)},
\[
  G(x)\,=\,\Erw_{x} G(S_{\sigma(y)})\ind{\sigma(y)<\tau}
\]
holds for all for all $0<x<y$.
\end{Lemma}

\begin{proof}
By \eqref{eq:harmonic V(x)}, $(G(S_{\tau \wedge n})_{n\ge 0}$ is a martingale. Hence, the optional sampling theorem implies
\begin{align*}
G(x)\ &=\ \Erw_{x} G(S_{\sigma(y)\wedge \tau \wedge n})\\
&=\ \Erw_{x} G(S_{\sigma(y)} \ind{\sigma(y)< \tau \wedge n}\,+\,\Erw_{x} G(S_{n}) \ind{n < \sigma(y) \wedge \tau}
\end{align*}
for all $0 < x<y$ and $n\in\N$. As $n\to\infty$, we have
$$ \Erw_{x}G(S_{\sigma(y)} \ind{\sigma(y)< \tau \wedge n}\ \to\ \Erw_{x} G(S_{\sigma(y)} \ind{\sigma(y)< \tau} $$
by the monotone convergence theorem, and
$$ \Erw_{x}G(S_{n}) \ind{n < \sigma(y) \wedge \tau}\ \le\ C(y+1) \Prob_{x}(n < \sigma(y) \wedge \tau)\ \to\ 0. $$
This completes the proof.
\end{proof}

Next are some asymptotic estimates involving the level $a$ overshoot $R_{a}=S_{\sigma(a)}-a$ of the random walk killed upon entering the positive halfline. As a by-product, another formula for $H$ is obtained.

\begin{Lemma}\label{lem:overshoot L_{1}}
Let $(S_{n})_{n\ge 0}$ be a centered random walk with $0<\Erw S_{1}^{2}<\infty$. Then for all $x > 0$, we have
\begin{gather}
\lim_{b\to\infty} \limsup_{a\to\infty} \Erw_{x} S_{\sigma(a)} \ind{\sigma(a)<\tau,R_{a}>b}\ =\ 0 \label{eqn:overshoot1}\\
\shortintertext{and}
\lim_{a\to\infty}\,\Erw_{x}R_{a}\1_{\{\sigma(a)<\tau\}}\ =\ 0. \label{eqn:overshoot2}
\end{gather}
\end{Lemma}

\begin{proof}
\textsc{Step 1} We first show that
\begin{equation}
  \label{eqn:firstStep}
  \lim_{a\to\infty}\,\Erw_{x} H(S_{\sigma(a)})\1_{\{\sigma(a)<\tau<\infty,\,S_{\sigma(a)}>(1+\eps)a\}}\ =\ 0
\end{equation}
for all $\eps>0$ and $x\ge 0$. As $H$ grows like the identity, we may replace $H(S_{\sigma(a)})$ with $S_{\sigma(a)}$. It is further no loss of generality to choose $x=0$. We then have
\begin{align}
\Erw S_{\sigma(a)}\1_{\{\sigma(a)<\tau,\,S_{\sigma(a)}>(1+\eps)a\}}\ &=\ \int_{(1+\eps)a}^{\infty}\Prob(\sigma(a)<\tau,S_{\sigma(a)}>y)\ dy\nonumber\\
&=\ \int_{(1+\eps)a}^{\infty}\int_{0}^{a}\Prob(S_{1}>y-x)\ \U_{a}(dx)\ dy\nonumber\\
&\le\ \int_{(1+\eps)a}^{\infty}\Prob(S_{1}>y-a)\ \U_{a}([0,a])\ dy\label{eq:crucial}
\end{align}
where
\begin{align*}
\U_{a}(dx)\ &=\ \sum_{n\ge 0}\Prob(S_{n}\in dx,\,0<S_{k}\le a\text{ for }k=0,\ldots,n)\\
&=\ \sum_{n\ge 0}\Prob(S_{n}\in dx,\,0<S_{n}-S_{n-k}\le a\text{ for }k=0,\ldots,n)\\
&=\ \sum_{n\ge 0}\sum_{k\ge 0}\Prob\left(\sigma_{k}=n,\,S_{\sigma_{k}}\in dx,\,S_{\sigma_{k}}\le a+\min_{0\le j\le\sigma_{k}}S_{j}\right)\\
&=\ \sum_{k\ge 0}\Prob\left(S_{\sigma_{k}}\in dx,\,S_{\sigma_{k}}\le a+\min_{0\le j\le\sigma_{k}}S_{j}\right)\\
&\le\ \sum_{k\ge 0}\Prob\left(S_{\sigma_{k}}\in dx\cap (0,a]\right).
\end{align*}
Since $\Erw S_{1}^{2}<\infty$ ensures $\Erw S_{\sigma_{1}}<\infty$, we infer that
$$ \U_{a}([0,a])\ \le\ \sum_{k}\Prob(S_{\sigma_{k}}\le a)\ \le\ Ca $$
for some $C>0$ and all $a\ge 1$. Returning to \eqref{eq:crucial}, we now obtain
\begin{align*}
\Erw S_{\sigma(a)}\1_{\{\sigma(a)<\tau,\,S_{\sigma(a)}>(1+\eps)a\}}\ &\le\ Ca\int_{(1+\eps)a}^{\infty}\Prob(S_{1}>y-a)\ dy\\
&\le\ \frac{C}{\eps}\int_{\eps a}^{\infty}y\,\Prob(S_{1}>y)\ dy
\end{align*}
and the last expression goes to 0 as $a\to\infty$ under the proviso $\Erw S_{1}^{2}<\infty$.

\vspace{.2cm}
\textsc{Step 2}. Next, we show that
\begin{equation*}
\lim_{b\to\infty}\limsup_{a\to\infty}\,\Erw_{x}H(S_{\sigma(a)})\1_{\{\sigma(a)<\tau<\infty,\,R_{a}>b\}}\ =\ 0
\end{equation*}
for all $x\ge 0$, thus proving \eqref{eqn:overshoot1}. Using the strong Markov property at time $\sigma(a/3)$, we have
\begin{align*}
\Erw_{x}H(S_{\sigma(a)})\1_{\{\sigma(a)<\tau<\infty,\,R_{a}>b\}}\ =\ \Erw_{x}H(S_{\sigma(a/3)})\Psi(S_{\sigma(a/3)})\1_{\{\sigma(a/3)<\tau<\infty\}},
\end{align*}
where
$$ \Psi(x)\ :=\ \Erw_{x}\left(\frac{H(S_{\sigma(a)})}{H(x)}\1_{\{\sigma(a)<\tau<\infty,\,b<R_{a}\le a\}}\right) $$
for $x>0$. Observe that
\begin{equation*}
\Psi(x)\ \le\ \Erw_{x}\left(\frac{H(S_{\sigma(a)})}{H(x)}\1_{\{\sigma(a)<\tau<\infty\}}\right)\ =\ \Prob_{x}^{\uparrow}(\sigma(a)<\infty)\ =\ 1,
\end{equation*}
where $\Prob_{x}^{\uparrow}$ denotes the harmonic transform with respect to $H$. Using this, we further obtain
\begin{align}
\begin{split}\label{eq:inequality step 2}
&\Erw_{x}H(S_{\sigma(a/3)})\Psi(S_{\sigma(a/3)})\1_{\{\sigma(a/3)<\tau<\infty\}}\\
&\le\ \Erw_{x}H(S_{\sigma(a/3)})\1_{\{S_{\sigma(a/3)}>2a/3\}}\\
&+\ \Erw_{x}H(S_{\sigma(a/3)})\1_{\{\sigma(a/3)<\tau<\infty,\,S_{\sigma(a/3)}\le 2a/3\}}\sup_{a/3\le y\le 2a/3}\Psi(y).
\end{split}
\end{align}
The first of the two terms on the right-hand side of this inequality converges to 0 as $a\to\infty$ by Step 1. As for the second one, we use that $H$ is harmonic and of linear growth together with Lemma \ref{lem:harmonicStopped} (which also holds for $H$ in the place  of $G$) to bound it by
\begin{align*}
\Erw_{x}H(S_{\sigma(a/3)})\1_{\{\sigma(a/3)<\tau<\infty\}}\sup_{a/3\le y\le 2a/3}\Psi(y)\ =\ H(x)\sup_{a/3\le y\le 2a/3}\Psi(y).
\end{align*}
Furthermore,
\begin{align*}
\sup_{a/3\le y\le 2a/3}\Psi(y)\ &\le\ \frac{H(2a)}{H(a/3)}\sup_{a/3\le y\le 2a/3}\Prob_{y}(\sigma(a)<\tau<\infty,\,b<R_{a}\le a)\\
&\le\ \frac{H(2a)}{H(a/3)}\sup_{0\le y\le a}\Prob(R_{a-y}>b)\\
&=\ (6+o(1))\,\sup_{y\ge 0}\Prob(R_{y}>b)\quad\text{as }a\to\infty.
\end{align*}
Consequently, recalling that $\Erw S_{1}^{2}<\infty$ implies the tightness of the overshoots $R_{a}$ for $a\ge 0$, the second term on the right-hand side of \eqref{eq:inequality step 2} converges to 0 as well when first letting $a$ and then $b$ tend to infinity.

\vspace{.1cm}
\textsc{Step 3}. In order to finally prove the last assertion of the lemma, we first note that, by another appeal to \eqref{eqn:firstStep}, it suffices to show
\begin{equation*}
\lim_{a\to\infty}\,\Erw_{x}R_{a}\1_{\{\sigma(a)<\tau<\infty,\,R_{a}\le a\}}\ =\ 0
\end{equation*}
for all $x>0$. Fix an arbitrary $\eps>0$. By Step 2 and \eqref{eq:growth of H}, we can pick $b>0$ so large that
$$ \limsup_{a\to\infty}\,\Erw_{x}H(S_{\sigma(a)})\1_{\{\sigma(a)<\tau<\infty,\,b<R_{a}\le a\}}\, <\,\frac{\eps}{2} $$
and thus also $a_{0}>0$ such that
$$ \Erw_{x}H(S_{\sigma(a)})\1_{\{\sigma(a)<\tau<\infty,\,R_{a}>b\}}\,<\,\eps $$
for all $a\ge a_{0}$. Consequently, as $a\to\infty$,
\begin{align*}
\Erw_{x}&R_{a}\1_{\{\sigma(a)<\tau<\infty,\,R_{a}\le a\}}\ \simeq\ \Erw_{x}H(R_{a})\1_{\{\sigma(a)<\tau<\infty,\,R_{a}\le a\}}\\
&\le\ b\,\Prob_{x}(\sigma(a)<\tau<\infty)\,+\,\Erw_{x}H(S_{\sigma(a)})\1_{\{\sigma(a)<\tau<\infty,\,R_{a}>b\}}\\
&\le\ o(1)\,+\,\eps
\end{align*}
which completes the proof.
\end{proof}

This result particularly implies the following identity for $H$ that will be useful below in the proof of Proposition \ref{prop:uniqueness H(x)}.

\begin{Cor}\label{cor:alternativeFormula}
For all $x > 0$, we have
$$ H(x)\,=\,\lim_{y\to\infty} y\,\Prob_{x}(\sigma(y)<\tau). $$
\end{Cor}

\begin{proof}
By Lemma \ref{lem:harmonicStopped}, for all $y>x>0$, we have
$$ H(x)\,=\,\Erw_{x} H(S_{\sigma(y)}) \ind{\sigma(y) \le \tau}. $$
Using \eqref{eq:growth of H}, for all $\epsilon>0$ and all $y$ large enough, we deduce
\begin{multline*}
(1-\epsilon) \left(y\,\Prob_{x}(\sigma(y)<\tau) + \Erw_{x} R_{y} \ind{\sigma(y)<\tau} \right)\\
\ \le\ H(x)\ \le\ y\,\Prob_{x}(\sigma(y)<\tau) + \Erw_{x} R_{y} \ind{\sigma(y)<\tau}.
\end{multline*}
Now use Lemma \ref{lem:overshoot L_{1}} upon letting $y\to\infty$ and then $\epsilon\to0$ to arrive at the assertion.
\end{proof}

We are now ready to give the proof of the main result of this section.

\begin{proof}[Proof of Proposition \ref{prop:uniqueness H(x)}]
Our proof follows along the same lines as the original one by Choquet and Deny \cite{Choquet+Deny:60}. Note first that, for all $A>0$, the function $G+AH$ satisfies the same assumptions as $G$ and is nonnegative on $[1,\infty)$ for large enough $A$. Therefore, we may assume without loss of generality that $G$ is bounded from below.

We consider the following regularization of the function $G$. For $\delta>0$ and $x\in\R$, put
\begin{equation*}
G^{\delta}(x)\,:=\,\frac{1}{\delta}\int_{x}^{x+\delta} G(z) \dd z.
\end{equation*}
The function $G^{\delta}$ is differentiable, and by assumption its derivative satisfies
\begin{equation*}
|(G^{\delta})'(x)|\,=\, |G(x+\delta)-G(x)|\,\le\,2C (1+x_{+}+\delta).
\end{equation*}
for some $C>0$. As a consequence, $x\mapsto G^{\delta}(x)/(1+x_{+})$ is uniformly continuous and bounded. Hence, by the Arzela-Ascoli theorem, there exist $0<y_{n}\uparrow\infty$ such that $x \mapsto G^{\delta}(x+y_{n})/(1+(x+y_{n})_{+})$ converges, uniformly on compact sets, to a bounded and continuous limit denoted as $\kappa^{\delta}$. The $y_{n}$ may further be chosen from $d\N$ if the random walk is $d$-arithmetic. The next argument shows this function to be harmonic for the random walk (without killing).

Indeed, as $G(x)=0$ for $x\le 0$, we infer from \eqref{eq:harmonic V(x)} that
\begin{gather}
G(x)\ =\ \Erw G(x+S_{1})\ind{x+S_{1}>0}\ =\ \Erw G(x+S_{1})\nonumber
\intertext{and thereupon, by Fubini's theorem,}
G^{\delta}(x)\ =\ \Erw G^{\delta}(x + S_{1})\quad\text{for all }x>0.\label{eq:harmonic G^delta}
\end{gather}
As a consequence,
\begin{align}
\kappa^{\delta}(x)\ &=\ \lim_{n\to\infty} \frac{G^{\delta}(x + y_{n})}{1+(x+y_{n})_{+}}\ =\  \lim_{n\to\infty} \frac{G^{\delta}(x+y_{n})}{y_{n}}\label{eqn:ab}\\
&=\ \lim_{n\to\infty}\frac{\Erw G^{\delta}(x+y_{n}+S_{1})}{y_{n}}\ =\ \Erw\left(\lim_{n\to\infty} \frac{G^{\delta}(x+y_{n}+S_{1})}{y_{n}}\right)\nonumber\\
&=\ \Erw \kappa^{\delta} (x+S_{1}) \nonumber,
\end{align}
for all $x\in\R$, having used \eqref{eq:harmonic G^delta}, then the domination assumption on $G$, and finally the dominated convergence theorem. This proves that $\kappa^{\delta}$ is indeed harmonic for the random walk $(S_{n})_{n\ge 0}$ and thus, by Lemma \ref{lem:CD-lemma G}, either a $d$-periodic continuous function or a constant.

As the next step, we show that
\begin{equation}\label{eqn:nextStep}
G^{\delta}(x)\ =\ \kappa^{\delta}(x) H^{\delta}(x)\,+\,\Erw_{x} G^{\delta}(S_\tau).
\end{equation}
for all $x>0$. First, writing \eqref{eq:harmonic G^delta} as $G^{\delta}(x)=\Erw_{x}G^{\delta}(x+S_{1\wedge\tau})$ for all $x>0$, it follows that $(G^{\delta}(x+S_{n\wedge\tau}))_{n\ge 0}$ forms a martingale and then as in the proof of Lemma \ref{lem:harmonicStopped} that
\begin{equation*}
G^{\delta}(x)\ =\ \Erw_{x}G^{\delta}(S_{\sigma(y)\wedge\tau})\ =\ \Erw_{x}G^{\delta}(S_{\sigma(y)})\ind{\sigma(y)<\tau}\,+\,\Erw_{x} G^{\delta}(S_{\tau}) \ind{\tau<\sigma(y)}
\end{equation*}
for all $0<x<y$. Observe that $\Erw_{x}G^{\delta}(S_{\tau}) \ind{\tau< \sigma(y)}=\Erw_{x} G^{\delta}(S_\tau)$ as $y\to\infty$.

\vspace{.1cm}
Compactly uniform convergence of $G^{\delta}(x+y_{n})/y_{n}$ to $\kappa^{\delta}$ will now be utilized to compute $\lim_{n\to\infty}\Erw_{x} G^{\delta} (S_{\sigma(y_{n})})\ind{\sigma(y_{n})<\tau}$ by bounding it from above and from below separately. Let $\epsilon,K> 0$ and choose $n$ large enough such that
\begin{equation}\label{eqn:unifConv}
\sup_{x \in [0,K]} |G^{\delta}(x+y_{n})/y_{n} - \kappa^{\delta}(x)| \le \epsilon.
\end{equation}
Then
\begin{align*}
\Erw_{x}G^{\delta}(S_{\sigma(y_{n})})\ind{\sigma(y_{n})<\tau}\ &\ge\ \Erw_{x}G^{\delta} (S_{\sigma(y_{n})}) \ind{\sigma(y_{n})<\tau,\,R_{y_{n}}\le K}\\
&\ge\ (\kappa^{\delta}(x)-\epsilon) y_{n}\,\Prob_{x}\big( \sigma(y_{n})<\tau,R_{y_{n}}\le K\big),
\end{align*}
using that $\kappa^{\delta}(S_{\sigma(y_{n})}-y_{n}) = \kappa^{\delta}(x)$ $\Prob_{x}$-a.s. by the periodicity or constancy of $\kappa^{\delta}$ (recall here that the $y_{n}$ are all chosen from $d\N$ if the random walk has lattice-span $d>0$). By letting $n\to\infty$ and use of Corollary \ref{cor:alternativeFormula}, this yields
\begin{multline*}
\liminf_{n\to\infty}\Erw_{x} G^{\delta} (S_{\sigma(y_{n})})\ind{\sigma(y_{n})<\tau} \ge\\
(\kappa^{\delta}(x)-\epsilon)\left(H(x)\,-\,\limsup_{a\to\infty}\,a\,\Prob_{x}\left(\sigma(a)<\tau, R_{a}>K\right)\right),
\end{multline*}
and thereupon, with the help of \eqref{eqn:overshoot1},
$$ \liminf_{n\to\infty}\Erw_{x}G^{\delta}(S_{\sigma(y_{n})})\ind{\sigma(y_{n})<\tau}\ge H(x) \kappa(x). $$
when letting $K\to\infty$ and then $\epsilon\to 0$.

\vspace{.1cm}
For the upper bound, we obtain by proceeding similarly
\begin{multline*}
\Erw_{x} G^{\delta} (S_{\sigma(y_{n})}) \ind{\sigma(y_{n})<\tau}\ \le\ (\kappa^{\delta}(x)+\epsilon)\,(y_{n}+K)\,\Prob_{x}\left( \sigma(y_{n})<\tau,R_{y_{n}}\le K\right)\\
+\ 2C\,\Erw_{x} S_{\sigma(y_{n})} \ind{\sigma(y_{n})<\tau, R_{y_{n}}>K},
\end{multline*}
for sufficiently large $n$, where $G^{\delta}(y) \le 2Cy$ for sufficiently large $y$ has been utilized. Now, by another use of Corollary \ref{cor:alternativeFormula}, \eqref{eqn:overshoot1} and also
$$ \lim_{n\to\infty}K\,\Prob_{x}\left( \sigma(y_{n})<\tau,R_{y_{n}}\le K\right) = 0, $$
we find
$$  \limsup_{n\to\infty}\Erw_{x} G^{\delta} (S_{\sigma(y_{n})})\ind{\sigma(y_{n})<\tau} \le H(x) \kappa^\delta(x) $$
upon letting $n\to\infty$, then $K\to\infty$ and finally $\epsilon\to0$. This completes the proof of \eqref{eqn:nextStep}.

Finally, we observe that the right continuity of $G$ implies $G^{\delta}(x)\to G(x)$ as $\delta\to0$ and in combination with $G^{\delta}(z)=0$ for $z < -\delta$ also
$$ |\Erw_{x} G^{\delta}(S_\tau)|\ \le\ \sup_{z \in [0,\delta]} |G(z)|\ \to\ 0\quad\text{as }\delta\to 0. $$
Consequently, by letting $\delta\to 0$ in \eqref{eqn:nextStep}, we conclude that $\kappa^{\delta}$ converges as well. Its limit $\kappa$ is also $d$-periodic or constant, right-continuous and satisfies $\kappa(x) = G(x)/H(x)$ for all $x > 0$. This finishes the proof.
\end{proof}

\section{The boundary case: proof of Theorem~\ref{thm:mainDerivative}}\label{sec:derivative}

In essence, the same techniques as those used in Section~\ref{sec:nonBoundary} can be used to prove Theorem~\ref{thm:mainDerivative}. However, additional complications arise because, as predicted by the theorem,
\begin{align*}
-\log M(t)\ =\ h(t)t^{\alpha}Z,
\end{align*}
is no longer integrable. As a consequence, it is impossible to directly give an analog of the function $F$ here. Instead, we will have to work with a truncated version of the \BRW\ that only includes individuals in the tree that never went ``too high''.

\vspace{.1cm}
Again, let $f\in\cS(\cM)$ be an arbitrary solution to Eq.~\eqref{eq:functional FPE} and $a>0$. For $n\ge 0$, define
\begin{equation}\label{eqn:defMnshaved}
M_{n}^{(a)}(t)\ :=\ \prod_{\substack{|v|=n\\ tL^{*}(v)<a}}f(tL(v)),
\end{equation}
where $L^{*}(v):=\max_{k\le |v|}L(v(k))$. By another appeal to the branching property of the \BRW, it follows immediately that $(M_{n}^{(a)}(t))_{n\ge 0}$ constitutes a bounded submartingale and therefore converges a.s.~to a limit that we denote by $M^{(a)}(t)$.

\vspace{.1cm}
Recall that assumption \eqref{eqn:ncsDerivative} ensures $\lim_{n\to\infty} \sup_{|v|=n}L(v)=0$ a.s.~and thus $\sup_{v\in\V}L(v)<\infty$. As a consequence, 
\[M^{(a)}(t)\,=\,M(t)\quad\text{a.s. for all }0 \le t<a/\sup_{v\in\V}L(v).\]
In particular, $M^{(a)}(t)$ is positive with positive probability for $a$ large enough.

\vspace{.1cm}
For $t\ge 0$, $a>1$ and $n\in\N_{0}$, we now define
\begin{align*}
Z_{n}^{(a)}(t)\ :=\ \sum_{|v|=n}(tL(v))^{\alpha}H(-\log(tL(v)/a))\,\1_{\{tL^{*}(v)<a\}}
\end{align*} 
where $H(x)=-\Erw S_{\tau(-x)}=x-\Erw_{x}S_{\tau}$ for $x>0$ is the essentially unique right-continuous harmonic solution to \eqref{eq:harmonic V(x)}. The last property entails that $(Z_{n}^{(a)}(t))_{n\ge 0}$ constitutes a nonnegative martingale, which is sometimes called the trimmed martingale for obvious reasons. It was introduced in \cite{BigKyp:04} as a positive martingale with an asymptotic behavior that can be linked to the derivative martingale. In particular, it converges a.s. and in $L^{1}$ to a limit $Z^{(a)}(t)$ under the assumptions of Theorem \ref{thm:mainDerivative} (see \cite[Proposition~A.3]{Aidekon:13}). By similar arguments as the ones above for the martingale $M^{(a)}(t)$ in combination with $H(x)=0$ for $x\le 0$ and $H(x)\simeq x$ as $x\to\infty$, we see that, if $t\ge 0$ and $a>t\,\sup_{v\in\V}L(v)$, then almost surely
\begin{align*}
Z^{(a)}(t)\ &=\ \lim_{n\to\infty}\sum_{|v|=n}(tL(v))^{\alpha}H(-\log(tL(v)/a))\\
&=\ \lim_{n\to\infty}\sum_{|v|=n}(tL(v))^{\alpha}(-\log(tL(v)/a))\\
&=\ \lim_{n\to\infty}\sum_{|v|=n}(tL(v))^{\alpha}(-\log L(v)) + t^{\alpha}\log(a/t) \sum_{|v|=n}L(v)^{\alpha}\ =\ t^{\alpha}Z,
\end{align*}
where we have also used that the additive martingale $\sum_{|v|=n}L(v)^{\alpha}$, $n\ge 0$, converges to $0$ a.s.~as $n\to\infty$ (see \cite{Lyons:97}).

\vspace{.1cm}
Using these observations, we now prove the following tameness result and counterpart of Lemma~\ref{lem:tameness} in the boundary case.

\begin{Lemma}\label{lem:tamenessBoundary}
Under the assumptions of Theorem~\ref{thm:mainDerivative}, for any function $f\in\cS(\cM)$ there exists $0 < C < \infty$ such that
\begin{equation}\label{eqn:tamenessBoundary}
\sup_{0<t\le 1}\frac{\log f(t)}{t^{\alpha}\log t}\ \le\ C.
\end{equation}
\end{Lemma}

\begin{proof}
Proceeding in a similar manner as in the proof of Lemma~\ref{lem:tameness},
we prove that, given any $f\in\cS(\cM)$, failure of \eqref{eqn:tamenessBoundary}, that is
\begin{equation}\label{eq:failure of tamenessBoundary}
\limsup_{t\to 0}\frac{\log f(t)}{t^{\alpha}\log t}\ =\ \infty,
\end{equation}
entails $M(t)\le \1_{\{Z=0\}}$ a.s.~for all $t,a>0$ and thus $f(0+)< 1$ which is impossible. We confine ourselves to the main steps of the proof as technical details are very similar to those in the proof of Lemma~\ref{lem:tameness}.

\vspace{.1cm}
Given \eqref{eq:failure of tamenessBoundary}, we can find two decreasing null sequences $(\vth_{n})_{n\ge 1}$ and $(\rho_{n})_{n\ge 1}$ such that, for all $n\in\N$ and $x\in [\rho_{n}\vth_{n},\vth_{n}]$,
\begin{align*}
-\log f(x)\,\ge\,n x^{\alpha}(-\log x).
\end{align*}
We then bound the conditional expectation of $M^{(a)}(t)$ given $\cG_{\vth_{n}/t}$. Namely, by the branching property of the \BRW,
\begin{align*}
M^{(a)}(t)\ &=\ \lim_{n\to\infty}\prod_{\substack{v\in\Upsilon_{\vth_{n}/t}\\ tL^{*}(v)<a}}f(tL(v))\ =\ \exp\Bigg(-\lim_{n\to\infty}\sum_{\substack{v\in\Upsilon_{\vth_{n}/t}\\ tL^{*}(v)<a}}\hspace{-.2cm}-\log f(tL(v))\Bigg).
\end{align*}
Bounding $-\log f(x)$ by $nx^{\alpha}(-\log x)$ if $x\in [\rho_{n}\vth_{n},\vth_{n}]$, and by $0$ otherwise, we then obtain
\begin{align}\label{eq:bound M^a(t)}
M^{(a)}(t)\ \le\ \exp\Bigg(-\limsup_{n\to\infty}\ n\hspace{-.5cm}\sum_{\substack{v\in\Upsilon_{\vth_{n}/t}\\ tL(v)\ge\rho_{n}\vth_{n},\,tL^{*}(v)<a}}\hspace{-.1cm} (tL(v))^{\alpha}(-\log tL(v))\Bigg).  
\end{align}
On the other hand, as $\bigvee_{n\ge 1} \cG_{\vth_{n}/t} =\cF_\infty$, we infer
\begin{align*}
Z^{(a)}(t)&\ =\ \lim_{n\to\infty}\Erw\big(Z^{(a)}(t)|\cG_{\vth_{n}/t}\big)\\
&\ =\ \lim_{n\to\infty}\sum_{\substack{v\in\Upsilon_{\vth_{n}/t}\\ tL^{*}(v)<a}}(tL(v))^{\alpha}H\big(-\log(tL(v)/a)\big).
\end{align*}
Therefore, by another use of $\lim_{n\to\infty}\max_{|v|=n}L(v)=0$ and $H(x) \simeq x$ as $x \to\infty$, we obtain for all $t\ge 0$
\begin{align*}
&\lim_{n\to\infty} \sum_{\substack{v\in\Upsilon_{\vth_{n}/t}\\ tL^{*}(v)<a}}(tL(v))^{\alpha}(-\log (tL(v)))\\
&=\ \lim_{n\to\infty}\left(\sum_{v\in\Upsilon_{\vth_{n}/t}}(tL(v))^{\alpha}H(-\log (tL(v)/a))\ +\ \log a\sum_{v\in\Upsilon_{\vth_{n}/t}}L(v)^{\alpha}\right)\\
&=\ Z^{(a)}(t)\quad \text{a.s.}
\end{align*}
Next, the many-to-one lemma provides us with (recall $\Prob_{t}=\Prob(\cdot|S_{0}=-\log t)$)
\begin{align*}
&\Erw\Bigg(\sum_{\substack{v\in\Upsilon_{\vth_{n}/t}\\ tL(v)\le\rho_{n}\vth_{n},\,tL^{*}(v)<a}} (tL(v))^{\alpha}(-\log tL(v))\Bigg)\\
&\hspace{.5cm}=\ t^{\alpha}\,\Erw(S_{\sigma(\log(t/\vth_{n}))}-\log t)\1_{\{R_{\log(t/\vth_{n})}\ge-\log\rho_{n},\,\sigma(\log(t/\vth_{n}))\le\tau(\log(t/a))\}}\\
&\hspace{.5cm}=\ t^{\alpha}\,\Erw_{t}S_{\sigma(\log(1/\vth_{n}))}\1_{\{R_{\log(1/\vth_{n})}\ge-\log\rho_{n},\,\sigma(\log(1/\vth_{n}))\le\tau(\log(1/a))\}},
\end{align*}
which by Lemma \ref{lem:overshoot L_{1}} converges to $0$ as $n\to\infty$.

\vspace{.1cm}
By combining the previous facts, we obtain
\begin{align*}
\lim_{n\to\infty} \sum_{\substack{v\in\Upsilon_{\vth_{n}/t}\\ tL(v)\ge\rho_{n}\vth_{n},\,tL^{*}(v)<a}}\hspace{-.1cm} (tL(v))^{\alpha}(-\log tL(v))\ =\ Z^{(a)}(t) \quad \text{a.s.}
\end{align*}
and therefore with the help of \eqref{eq:bound M^a(t)} that $M^{(a)}(t)\le\1_{\{Z^{(a)}(t)=0\}}$ for all $a,t>0$. But the latter entails $M(t)\le\1_{\{Z=0\}}$ upon letting $a\to\infty$ which is impossible because it would imply $f(t)=\Prob(Z=0)<1$ for all $t>0$, a contradiction to $f(0+)=1$.
\end{proof}

Using the tameness assumption for fixed points of the smoothing transform, we now can identify the function defined by
\begin{equation}
  \label{eqn:defineHa}
  F^{(a)}(t) = \Erw\left(- \log M^{(a)}(t)\right).
\end{equation}
We prove this function to be harmonic for the random walk $(S_{n})_{n\ge 0}$ killed when hitting $(-\infty,0]$ and  therefore to be a multiple of $H$.

\begin{Lemma}\label{lem:identifyHa}
Let $r\ge 1$ be the span of $T$. Assuming \eqref{eqn:ncsDerivative}, for any function $f \in \mathcal{S}(\mathcal{M})$ with disintegration $M(t)$ and all $a > 0$, there exists a function $h^{(a)}$, multiplicatively $r$-periodic if $r>1$ and constant otherwise, such that
$$ F^{(a)}(t)\,=\,t^{\alpha}\,H(-\log (t/a))\,h^{(a)}(t) $$
for all $t\ge 0$.
\end{Lemma}

\begin{proof}
The proof follows similar lines as the one of Lemma~\ref{lem:identificationOfH}. We prove that the function $F^{(a)}$ defined in \eqref{eqn:defineHa} is related to a harmonic function for the random walk conditioned to stay positive which in turn allows us to characterize it up to multiplication by a multiplicatively $r$-periodic or constant function.

Recalling that
\begin{align*}
-\log M^{(a)}(t)\ =\ \lim_{n\to\infty} \sum_{|v|=n} -\log f(tL(v))\1_{\{tL^{*}(v)<a\}}\quad\text{a.s.}
\end{align*}
for all $a\le 1$ and $t>0$, Lemma~\ref{lem:tamenessBoundary} implies
\begin{align}\label{eqna}
0\ \le\ -\log M^{(a)}(t)\ \le\ C Z^{(a)}(t)\quad\text{a.s.}
\end{align}
Consequently, as $(Z^{(a)}_{n}(t))_{n\ge 0}$ is uniformly integrable, we infer that
\begin{align*}
F^{(a)}(t)\ =\ \Erw\left(-\log M^{(a)}(t)\right)\ \in\ [0, C t^{\alpha}H(-\log (t/a))].
\end{align*}
for all $t\le a$.

By conditioning with respect to $\mathcal{F}_{1}$ and use of the many-to-one lemma, it follows that 
\begin{align*}
F^{(a)}(t)\ &=\ \Erw\left(\sum_{|u|=1} F^{(a)}(tL(u))\1_{\{t L(u)<a\}}\right)\\
&=\ \Erw\left(F^{(a)}(te^{-S_{1}})e^{\alpha S_{1}}\1_{\{S_{1} - \log t > -\log a\}}\right)
\end{align*}
for all $t\le a$. Hence, the function $g^{(a)}(x):=a^{-\alpha}e^{\alpha x}F^{(a)}(ae^{-x})$ satisfies
$$ g^{(a)}(x)\ =\ \Erw g^{(a)}(x+S_{1}) \1_{\{x+S_{1} >0\}} $$
for all $x>0$, and it is right-continuous because $f$ is left-continuous, by dominated convergence. Furthermore, \eqref{eqna} implies that 
$$ g^{(a)}(x)\ \le\ Ca^{-\alpha} e^{\alpha x} \Erw Z^{(a)}(ae^{-x})\ \le\ C H(x), $$
and thus the required boundedness of $(1+x)^{-1}g^{(a)}(x)$. Invoking Proposition \ref{prop:uniqueness H(x)}, we conclude that $g^{(a)}$ equals $\kappa^{(a)}H$ for some $\kappa^{(a)}$ which is $\log r$-periodic if $T$ has span $r\ne 1$, and is constant otherwise. 
The proof is completed by rewriting this result in terms of $F^{(a)}(t)$ and putting $h^{(a)}(t):=\kappa^{(a)}(-\log(t/a))$.
\end{proof}

With the help of the last lemma, we are now able to given an explicit expression for $M^{(a)}(t)$ and thereby to find the value of $M(t)$ upon letting $a\to\infty$.

\begin{proof}[Proof of Theorem~\ref{thm:mainDerivative}]
By the branching property of the \BRW, we have almost surely
\begin{align*}
&\Erw\left(-\log M^{(a)}(t)\middle|\cF_{n}\right)\ =\ \sum_{|v|=n}F^{(a)}(tL(v))\,\1_{\{tL^{*}(v)<a\}}\\
&\quad=\ h^{(a)}(t)\sum_{|v|=n}(tL(v))^{\alpha}\,H(-\log (tL(v)/a))\,\1_{\{tL^{*}(v)<a\}}\ =\ h^{(a)}(t)Z_{n}^{(a)}(t)
\end{align*}
for all $n\in\N$. Letting $n\to\infty$, this yields
\begin{align*}
M^{(a)}(t)\ =\ e^{-h^{(a)}(t) Z^{(a)}(t)} \quad \text{a.s.}
\end{align*}
Therefore, for all $a$ large enough and all $t \in [0,a/\sup_{v\in\V}L(v)]$,
$$ M(t)\ =\ M^{(a)}(t) = e^{-h^{(a)}(t)t^{\alpha} Z}\quad\text{a.s} $$
In particular, as $h^{(a)}$ is multiplicatively $r$-periodic or constant, we infer that $h^{(a)}=h$ does not depend on $a$ for $a$ large enough. Since
$$ \Psi(t)\ =\ \Erw\left(M(t)\right)\ =\ \Erw\left(e^{-h(t)t^{\alpha}Z}\right), $$
the proof is complete when finally noting that $h\in\mathcal{H}_{r}$ follows from the fact that $\Psi(t)$ is nonincreasing.
\end{proof}

\section{Fixed points of the smoothing transform and fractal measures on the boundary of the \BRW}\label{sec:measure}

The purpose of this supplementary section is to show that any fixed point of the smoothing transform \eqref{eqn:smoothingTransform} can be thought of as the total mass of a random fractal measure on the boundary of the associated \BRW. More precisely, this connection is established by a one-to-one map between these fixed points and random measures $\nu$ on the boundary $\partial\V$ of the tree (see below for details) such that, for all $u\in\V$,
$$  \frac{\nu\left(\left\{v\in\partial\V:v(|u|)=u\right\}\right)}{L(u)} $$
is independent of $\mathcal{F}_{|u|}$ and has the same law as $\nu(\V)$. Any random measure $\nu$ of this kind is called \emph{fractal}.

Let $X$ be a random variable with Laplace transform $f\in\cS(\cL)$, its law thus a fixed point of the smoothing transform. Then there exists a family $(X(v))_{v\in\V}$ of copies of $X$, defined on the same probability space as the multiplicative \BRW\ $L=(L(v))_{v\in\V}$ (possibly enlarged) such that
\begin{equation}\label{eqn:coupling}
X(v)\ =\ \sum_{j\ge 1}T_{j}^{v}X(vj)\ =\ \sum_{j\ge 1}\frac{L(vj)}{L(v)} X(vj)
\end{equation}
for all $v\in\V$. Namely, let $\{X^{(n)}(v):|v|=n\}$ for any $n\in\N$ denote a family of independent copies of $X$ which are also independent of $\{L(v):|v|=n\}$. For $v\in\V$ with $|v| < n$, we then define recursively
\begin{align*}
X^{(n)}(v)\ =\ \sum_{j\ge 1}\frac{L(vj)}{L(v)}X^{(n)}(vj).
\end{align*}
As $X$ satisfies \eqref{eqn:smoothingTransform} and by the branching property of the \BRW, we see that each $X^{(n)}(v)$ is a copy of $X$ and depends only on the variables defined on the subtree rooted at vertex $v$. The existence of $(X(v))_{v\in\V}$ with the claimed properties is now ensured by Kolmogorov's consistency theorem because the laws of
$$ \{(X^{(n)}(v),L(v)):|v|\le n\},\quad n\in\N $$
constitute a projective familiy.

\vspace{.1cm}
Recall that $\partial\V=\N^\N$ denotes the boundary of the tree $\V$ and becomes a complete metric space when endowed with the ultrametric distance
$$ d(u,v)\,=\,\exp(-\min\{k\ge 1:u_{k}\ne v_{k}\}). $$
Putting
\begin{align*}
B(u)\,:=\,\left\{v\in\partial\V:v(|u|)=u\right\}
\end{align*}
for $u\in\V$, the family $(B(u))_{u\in\V}$ forms a basis of the topology on $\partial\V$ and its Borel $\sigma$-field. 

\vspace{.1cm}
With the help of the familiy $(X(v))_{v\in\V}$ introduced above, a one-to-one map between the fixed points of the smoothing transform and the random fractal measures on $\partial\V$ can now be constructed as follows. Observe that, for any such $\nu$, the total mass $\nu(\partial\V)$ is a fixed point of the smoothing transform associated with $L$. This follows because, by $\sigma$-additivity of $\nu$,
\begin{align*}
\nu(\partial\V)\ =\ \nu\Bigg(\bigcup_{|v|=1}B(v)\Bigg)\ =\ \sum_{|v|=1}\nu(B(v))\ =\ \sum_{|v|=1}L(v)\frac{\nu(B(v))}{L(v)},
\end{align*}
and the fractal property of $\nu$.

\vspace{.1cm}
Conversely, the above construction allows us to define a fractal measure $\nu$ for each fixed point of the smoothing transform such that the law of $\nu(\partial\V)$ equals this fixed point. 
Namely, with $(X(v))_{v\in\V}$ as defined above, we put
\begin{align*}
\nu(B(v))\,:=\,X(v)L(v)
\end{align*}
for any $v\in\V$. By \eqref{eqn:coupling}, this provides a well-defined consistent $\sigma$-additive measure on $\N^{k}$ for each $k\in\N$. Thus, by another use of Kolmogorov's consistency theorem, we can extend $\nu$ to a random measure on $\partial\V$, and it has the fractal property by definition as $X(v)$ is independent of $\cF_{|v|}$ for any $v$.

\vspace{.1cm}
Under the assumptions of Theorem~\ref{thm:mainNonBoundary} or~\ref{thm:mainDerivative}, the fractal measure $\nu$ can be even explicitly defined as a marked Poisson point process, namely
\begin{align*}
\nu\ =\ \sum_{j\ge 1}\xi_{j}\delta_{v^{j}},
\end{align*}
where $(\xi_{j}, v^{j})$ are the atoms of a bivariate Poisson point process on $\IRg\times\partial\V$ with intensity $\pi(\dd x)\mu_{\alpha}(\dd v)$. Here $\pi$ equals the L\'evy jump measure of a L\'evy process with characteristic exponent $h(t)t^{\alpha}$, $h$ the function associated with $f$ by the respective theorem.
Moreover, $\mu_{\alpha}$ denotes the random measure on $\partial\V$ defined by
\begin{align*}
\mu_{\alpha}(B(v))\ =\ \lim_{n\to\infty} \sum_{|u|=n,u\succ v}L(u)^{\alpha}
\end{align*}
in the regular case (assuming \eqref{eqn1:regular case} and \eqref{eqn2:regular case}), and by
\begin{align*}
\mu_{\alpha}(B(v))\ =\ \lim_{n\to\infty}\sum_{|u|=n,u\succ v}(-\log L(u))L(u)^{\alpha}
\end{align*}
in the boundary case (assuming \eqref{eqn:boundaryCase} and \eqref{eqn:ncsDerivative}) where the former definition would only give the null measure. Indeed, with $\mu_{\alpha}$ thus defined and Campbell's formula, we obtain that
\begin{align*}
\Erw e^{-t\nu(B(v))}\ =\ \Erw \exp(-t^{\alpha} h(t) \mu_\alpha(B(v)))\ =\ f(t),
\end{align*}
for all $v \in \V$ as expected.

\section*{Acknowledgments}

Most of this work was done during a visit of the second author in March 2019 at the University of M\"unster. Financial support and kind hospitality are gratefully acknowledged.

\section*{Data availability statement}

This manuscript was produced with no additional data.

\def\cprime{$'$}
\providecommand{\bysame}{\leavevmode\hbox to3em{\hrulefill}\thinspace}
\providecommand{\MR}{\relax\ifhmode\unskip\space\fi MR }
\providecommand{\MRhref}[2]{%
  \href{http://www.ams.org/mathscinet-getitem?mr=#1}{#2}
}
\providecommand{\href}[2]{#2}

\end{document}